\documentclass[12pt]{amsart}

\usepackage{amsthm,amscd,amsfonts}
\usepackage{amssymb, upref}
\usepackage{latexsym}
\usepackage{amsmath,amsthm,amsfonts,amssymb,extarrows}
\usepackage[pdftex]{graphicx}
\usepackage{tikz}
\usepackage{amscd}
\usepackage[all]{xy}
\usepackage{xy}

\newtheorem{theorem}{Theorem}[section]
\newtheorem{lemma}[theorem]{Lemma}

\newtheorem{proposition}[theorem]{Proposition}
\theoremstyle{definition}

\newtheorem{remark}[theorem]{Remark}

\def\bal{\begin{array}{ll}}
\def\eal{\end{array}}

\def\fii{\mathfrak i}
\def\diag{\mathrm{diag}}
\def\Cat{\textrm{Cat}}
\def\CAT{\textsc{Cat}}
\def\Mot{\textrm{Mot}}
\def\Sch{\textrm{Sch}}
\def\wt{\mathrm{wt}}
\def\ea{\mathrm{ea}}
\def\enor{\mathrm{enor}}
\def\valley{\mathrm{valley}}
\def\peak{\mathrm{peak}}
\def\bZ{\mathbb{Z}}

\def\bC{\mathbb{C}}
\def\cC{\mathcal{C}}
\def\fii{\mathfrak i}
\def\N{\textbf{N}}
\def\E{\textbf{E}}
\def\D{\textbf{D}}
\numberwithin{equation}{section}
\begin{document}
\title{Identities involving weighted Catalan, Schr\"oder and Motzkin paths}

\author{Zhi Chen}
\address{Department of Mathematics, Nanjing Agricultural University,
Nanjing 210095, People's Republic of China}
\email{chenzhi@njau.edu.cn}
\author{Hao Pan}
\address{Department of Mathematics, Nanjing University, Nanjing 210093,People's Republic of China}
\email{haopan79@zoho.com}
\thanks{The first author is supported by the Natural Science Foundation of Jiangsu Province (BK20160708) and the Fundamental Research Funds for the Central Universities (Y0201600166).}

\begin{abstract}
In this paper, we investigate the weighted Catalan, Motzkin and Schr\"oder numbers together with the corresponding weighted paths. The relation between these numbers is illustrated by three equations, which also lead to some known and new interesting identities. To show these three equations, we provide combinatorial proofs. One byproduct is to find a bijection between two sets of Catalan paths: one consisting of those with $k$ valleys, and the other consisting of $k\ \N$ steps in even positions.
\end{abstract}

\maketitle

\section{Introduction}
\setcounter{equation}{0}\setcounter{theorem}{0}

The {\it Catalan numbers}
\begin{equation*}
C_n:=\frac{1}{n+1}{2n \choose n}, \qquad n\geq 0,
\end{equation*}
play a very important role in combinatorics. In \cite[Exercise 6.19]{S99}, Stanley listed $66$ kinds of different combinatorial interpretations of $C_n$.
For example, $C_n$ counts all Calatan paths of order $n$, which are the plane lattice paths from $(0,0)$ to $(n,n)$ which consisting of up steps $\N:=(0,1)$ and horizontal steps $\E:=(1,0)$ and never going below the line $y=x$.
A classical result on the Catalan numbers is the Touchard's identity:
\begin{equation}\label{Touchard}
C_{n}=\sum_{k=0}^{\lfloor\frac{n-1}{2}\rfloor}\binom{n-1}{2k}
C_k2^{n-1-2k},
\end{equation}
where $\lfloor x\rfloor=\max\{n\in\bZ:\, n\leq x\}$.

Intimately related to Catalan numbers are {\it Narayana numbers}, which are defined by
\begin{equation*}
N_{n,k}:=\frac1{n}\binom{n}{k}\binom{n}{k+1},\quad 0\leq k\leq n-1.
\end{equation*}
In combinatorics, $N_{n,k}$ counts all Catalan paths of order $n$ having exactly $k$ valleys (\cite[Section 2.4.2]{P15}), and thus we have
\begin{align*}
C_n&=N(n,0)+N(n,1)+\cdots+N(n,k-1)\\
   &=\sum_{k=0}^{n-1}\frac1{n}\binom{n}{k}\binom{n}{k+1}.
\end{align*}

In \cite{C03}, Coker derived the following two identities using generating functions:
\begin{equation}\label{Coker1}
\sum_{k=0}^{n-1}\frac1{n}\binom{n}{k}\binom{n}{k+1}x^k=\sum_{k=0}^{\lfloor\frac{n-1}{2}\rfloor}\binom{n-1}{2k}C_k(1+x)^{n-1-2k},
\end{equation}
\begin{equation}\label{Coker2}
\sum_{k=0}^{n-1}\frac1{n}\binom{n}{k}\binom{n}{k+1}x^{2k}(1+x)^{2(n-1-k)}=\sum_{k=0}^{n-1}\binom{n-1}{k}
C_{k+1}x^k(1+x)^{k}.
\end{equation}
It is easy to see that the Touchard's identity (\ref{Touchard}) follows from (\ref{Coker1}) in the case when $x=1$.

Coker~\cite{C03} proposed the problems of finding combinatorial interpretations of these two identities. In~\cite{CYY08}, Chen, Yan and Yang answered Coker's problems and gave the bijective proofs of Coker's identities.
The key ingredient of their proofs is the weighted Motzkin paths. A {\it Motzkin path} of order $n$ is a lattice path from $(0,0)$ to $(n,n)$, which never passes below the line $y=x$ and consists of double up steps $\N_2=(0,2)$, double horizontal steps $\E_2:=(2,0)$ and diagonal steps $\D:=(1,1)$. The number of all Motzkin paths of order $n$ is counted by the {\it Motzkin number}
\begin{equation*}
M_n=\sum_{k=0}^{\lfloor\frac n2 \rfloor}{n\choose 2k} C_k, \quad n\geq 0.
\end{equation*}
The Motzkin numbers also have some nice combinatorial properties. For example, in \cite{DS77}, Donaghey and Shapiro listed $14$ kinds of different combinatorial interpretations of $M_n$.

A \textsl{weighted lattice path} of order $n$ is an $n$-ordered lattice path where some of the steps are assigned with weights. For example, we can obtain an {\it $(a,b)$-Motzkin path} by assigning weight $a$ to each $\D$ step and weight $b$ to each $\E_2$ step on a Motzkin path. Denote the set of all Motzkin paths and the set of all $(a,b)$-Motzkin paths of order $n$ by $\Mot(n)$ and $\Mot_n(a,b)$ respectively. The weight of a path $P$, denoted by $\wt(P)$, is the product of all the weights assigned on its steps. The weight of a set of paths is the sum of the weights of all the paths in the set. Define the {\it $(a,b)$-Motzkin number} $M_n^{(a,b)}$ to be the weight of the set $\Mot_n(a,b)$ and then we have
\begin{align}\label{Mnab}
M_n^{(a,b)}:=&\sum_{P\in \Mot_n(a,b)}\wt(P)
=\sum_{P\in \Mot_n(a,b)}a^{\diag(P)} b^{\ea_2(P)}\nonumber\\
=&\sum_{k=0}^{\lfloor n/2 \rfloor}C_k {n \choose 2k} a^{n-2k} b^k,
\end{align}
where statistics $\diag$ and $\ea_2$ respectively count the number of diagonal steps and double horizontal steps in the Motzkin path $P$.  To get equation~(\ref{Mnab}), consider the weight of the subset of $\Mot_n(a,b)$ consisting of paths with exactly $k$ $\N_2$'s steps, which equals $C_k {n \choose 2k} a^{n-2k} b^k$ since\medskip

\noindent(i) there are $(n-2k)$ diagonal steps for each path in the subset and each $\D$ has weight $a$,\medskip

\noindent(ii) there are ${n \choose n-2k}$ ways to insert diagonal steps in a Catalan path of order $2k$ to form a path in $\Mot(n)$,\medskip

\noindent(iii) there are $C_k$ such Catalan paths of order $2k$ with double up steps $\N_2$ and double horizontal steps $\E_2$ with weight $b$ on each one.\medskip

Clearly $M_n=M_n^{(1,1)}$.
Substituting $a=0$, $b=1$ in (\ref{Mnab}), we have
\begin{equation}\label{Mn01}
M_{n}^{(0,1)}=\begin{cases} C_{\frac n2},&\text{if }n\text{ is even},\\
0,&\text{if }n\text{ is odd}.
\end{cases}
\end{equation}
Furthermore, we also get an equivalent form of the Touchard identity by specializing $a=2$, $b=1$:
\begin{equation}\label{Mn21}
M_n^{(2,1)}=C_{n+1}.
\end{equation}

Motivated by (\ref{Mn01}) and (\ref{Mn21}), it is natural to ask whether there exists a common generalization of (\ref{Coker1}) and (\ref{Coker2}).
Here we  give such an identity in the following theorem using $(a,b)$-Motzkin numbers, which will be proved in section $2$.
\begin{theorem}\label{CMCorT}
Suppose that $n\geq 1$ and $a,b\in\bC$. Then
\begin{equation}
\label{CMCor}
\sum_{k=0}^{n-1}\frac{1}{n}{n \choose k}{n \choose k+1} x^{2k} b^{n-1-k}= \sum_{k=0}^{n-1}{n-1 \choose k} M_{k}^{(a,b)}x^k(x^2-ax+b)^{n-1-k}.
\end{equation}
\end{theorem}
Clearly (\ref{Coker1}) follows from (\ref{CMCor}) in the case when $a=0$, $b=1$. Noticing that $M_k^{(2(1+x),(1+x)^2)}=M_k^{(2,1)}(1+x)^k$ and by specializing  $a=2(1+x)$, $b=(1+x)^2$, we get (\ref{Coker2}).

Another important combinatorial sequence arising from the lattice path enumeration consists of the {\it Schr\"oder numbers} $S_n$, which counts the number of all Schr\"oder paths of order $n$. An $n$-ordered {\it Schr\"oder path} is a lattice path from $(0,0)$ to $(n,n)$, which never passes below the line $y=x$ and consist of up steps $\N=(0,1)$, horizontal steps $\E:=(1,0)$ and diagonal steps $\D:=(1,1)$.
We mention that $\binom{n+k}{k}C_{k}$ counts all Schr\"oder paths of order $n$ having exactly $k$ $\E$ steps, since\medskip

\noindent(i) a Schr\"oder paths of order $n$ has exactly $k$ $\E$ steps if and only if
it contains exactly $(n+k)$ steps,\medskip

\noindent(ii) there are $\binom{n+k}{n-k}$ ways to choose $(n-k)$ $\D$ steps,\medskip

\noindent(iii) removing all diagonal steps we get a Catalan path of order $k$.\medskip

Therefore we have
\begin{equation*}
S_n=\sum_{k=0}^n\binom{n+k}{k}C_{k}, \quad n\geq 0.
\end{equation*}
Furthermore, we also have the identity
$S_{n}=2M_{n-1}^{(3,2)}$ \cite{CW12, Y07}.
Motivated by (\ref{CMCor}), we obtain the following theorem.
\begin{theorem}\label{SMCorT}
Suppose that $n\geq 1$ and $a,b\in\bC$. Then
\begin{equation}\label{SMCor}
\sum_{k=0}^{n}\binom{n+k}{2k}C_k x^{2k}(b-x^2)^{n-k}=
b\sum_{k=0}^{n-1}\binom{n-1}{k}M_k^{(a,b)}x^k(x^2-ax+b)^{n-1-k}.
\end{equation}
\end{theorem}
Setting $a=0$, $b=1$ and replacing $x^2$ by $x$, we get a special case of~(\ref{SMCor}) as the following identity:
\begin{equation}\label{SMabx1}
\sum_{k=0}^{n}\binom{n+k}{2k}C_k x^{k}(1-x)^{n-k}=\sum_{k=0}^{\lfloor\frac{n-1}{2}\rfloor}\binom{n-1}{2k}
C_kx^{k}(1+x)^{n-1-2k}.
\end{equation}
Similarly, letting $a=2(1+x)$ and $b=(1+x)^2$, we can get
\begin{equation}\label{SMabx2}
\sum_{k=0}^{n}\binom{n+k}{2k}C_k x^{2k}(1+2x)^{n-k}=\sum_{k=0}^{n-1}\binom{n-1}{k}
C_{k+1}x^k(1+x)^{k+2}.
\end{equation}
In particular, setting $x=\frac12$ in (\ref{SMabx1}), we have
\begin{equation}
S_n=\sum_{k=0}^{\lfloor\frac{n-1}{2}\rfloor}\binom{n-1}{2k}
C_k2^{k+1}3^{n-1-2k},
\end{equation}
which is exactly the identity $S_n=2M_{n-1}^{(3,2)}$.

In order to show Theorem~\ref{CMCorT} and~\ref{SMCorT}, we need to use the weighted Catalan numbers and
the weighted Schr\"oder numbers. An {\it $(a,b)$-Schr\"oder path} of order $n$ is a Schr\"oder path $P$ of order $n$, whose
$\D$ steps are assigned with weight $a$ and $\E$ steps are assigned with weight $b$. We denote by $\Sch_n(a,b)$ the set of all $(a,b)$-Schr\"oder paths of order $n$.
Define the $n^{th}$ {\it $(a,b)$-Schr\"oder number} $S_n^{(a,b)}$ to be the weight of the set $\Sch_n(a,b)$ as follows
\begin{equation*}
S_n^{(a,b)}:=\sum_{P\in\Sch_n(a,b)}\wt(P).
\end{equation*}
Let $\ea(P)$ denote the numbers of $\E$ steps of $P$. Since the number of $n$-th Schr\"{o}der paths having $k$ diagonal steps is given by $\frac{1}{n}{n\choose k}{2n-k\choose n-1}$ for $0\leq k\leq n$~(\cite{BSS93}), we have
\begin{align}\label{Snab}
S_n^{(a,b)}&=\sum_{P\in\Sch_n(a,b)}a^{\diag(P)}b^{\ea(P)}\nonumber\\
           &=\sum_{k=0}^n\binom{n+k}{2k} C_{k}a^{n-k}b^{k}.
\end{align}

Assume that $P$ is a Catalan path of order $n$. We may write $P=p_1p_2\cdots p_{2n}$, where $p_1,\ldots,p_{2n}\in\{\N, \E\}$.
If $p_i=\E$ and $p_{i+1}=\N$, then $p_ip_{i+1}$ forms a valley of $P$.
By assigning each $\E$ step of $P$ which is followed immediately by an $\N$ step with weight $b$, and each other $\E$ step with weight $a$, we get the so called {\it valley type $(a,b)$-Catalan path}.
The set of all valley type $(a,b)$-Catalan paths of order $n$ is denoted by $\Cat_n(a,b)$. The corresponding $n^{th}$ {\it valley type $(a,b)$-Catalan number} $C_n^{(a,b)}$ is defined to be the weight of $\Cat_n(a,b)$ as follows
\begin{equation*}
C_n^{(a,b)}:=\sum_{P\in\Cat_n(a,b)}\wt(P).
\end{equation*}

Since the number of $n$-ordered Catalan paths having exactly $k$ valleys coincides with the Narayana number $N(n,k)$,
we get
\begin{align}\label{Cnab}
C_n^{(a,b)}&=\sum_{P\in\Cat_n(a,b)}a^{n-\valley(P)}b^{\valley(P)}\nonumber\\
           &=\sum_{k=0}^{n-1}\frac{1}{n}\binom{n}{k}\binom{n}{k+1}a^{n-k}b^{k},
\end{align}
where
\begin{equation*}
\valley(P)=\#\{1\leq i\leq n-1:\,p_i=\E,\ p_{i+1}=\N\}.
\end{equation*}

The main result in this paper is to establish the identities between $C_n^{(a,b)}$, $M_n^{(a,b)}$ and $S_n^{(a,b)}$, which are given in the following theorem.
\begin{theorem}\label{thm:CSM}
Suppose that $a,b\in\bC$, then we have
\begin{equation}\label{CM}
C_n^{(a,b)}= a M_{n-1}^{(a+b,ab)},
\end{equation}
\begin{equation}\label{SC}
b S_n^{(a,b)}= (a+b) C_n^{(b,a+b)},
\end{equation}
\begin{equation}\label{SM}
S_{n}^{(a,b)}= (a+b) M_{n-1}^{(a+2b,ab+b^2)}.
\end{equation}
\end{theorem}
It is worth mentioned that Theorem~\ref{thm:CSM} not only generalizes the known facts that $C_n=M_{n-1}^{(2,1)}$ and $S_n=2M_{n-1}^{(3,2)}$, but also can be used to show Theorem \ref{CMCorT} and \ref{SMCorT}.
Furthermore, as we shall see later in Section 3, more explicit expressions involving $C_n$, $S_n$ and $M_n$ can be deduced from Theorem \ref{thm:CSM}.

This paper is organized as follows.
In Section 2, we prove how Theorem \ref{thm:CSM} implies  Theorems \ref{CMCorT} and \ref{SMCorT}.
In Section 3, we give some applications of Theorems \ref{CMCorT}-\ref{thm:CSM}, which lead to some known and new identities.
Sections 4 and 5 are the main part of this paper. In these two sections we are going to give combinatorial proofs for (\ref{CM}) and (\ref{SC}) in Theorem~\ref{thm:CSM}, which lead to (\ref{SM}). We first give a bijection between the set of Catalan paths with $k$ valleys and that with $k$ $\N$ steps in even positions in Section 4. This allows us to give another combinatorial interpretation of $C_n^{(a,b)}$. In Section 5, we give the combinatorial proof of Theorem \ref{thm:CSM}.

\section{Proofs of Theorem \ref{CMCorT} and \ref{SMCorT} using Theorem \ref{thm:CSM}}
\setcounter{equation}{0}\setcounter{theorem}{0}

In this section, we prove Theorems \ref{CMCorT} and \ref{SMCorT} under the assumption that Theorem \ref{thm:CSM} holds. To do that, we need the following lemma.
\begin{lemma}\label{Mna12L}
For $n\geq 0$ and $a_1,a_2,b\in\bC$,
\begin{equation}\label{Mna12}
M_n^{(a_1+a_2,b)}=\sum_{k=0}^n\binom{n}{k}M_k^{(a_1,b)}a_2^{n-k}.
\end{equation}
\end{lemma}
\begin{proof}
Let $\Mot_n(\{a_1,a_2\},b)$ denote the set of all Motzkin paths of order $n$, where each $\E_2$ step is assigned with weight $b$ and each $\D$ step is assigned with weight either $a_1$ or $a_2$. The paths in $\Mot_n(\{a_1,a_2\},b)$ are called $(\{a_1,a_2\},b)$-Motzkin paths.

For an $(a_1+a_2, b)$-Motzkin path $P$ where each $\D$ step is assigned with weight $a_1+a_2$, by splitting into the summation of $a_1$ and $a_2$ we can assign each $\D$ step by the weight either $a_1$ or $a_2$. In this way we get $2^{\diag(P)}$ $(\{a_1,a_2\},b)$-Motzkin paths.

Consider the set of all $(\{a_1,a_2\},b)$-Motzkin paths of order $n$.
There are $\binom{n}{n-k}$ ways to choose $n-k$ $\D$ steps among all the $n$ steps, which are assigned with weight $a_2$ for each one. Each of the rest  $\D$ steps is assigned with weight $a_1$. Each $\E_2$ step is assigned with weight $b$.
Thus we can reduce a $k$-ordered $(a_1, b)$-Motzkin paths by removing all $\D$ steps weighted by $a_2$ and get
\begin{align*}
M_n^{(a_1+a_2,b)}=&\sum_{P\in\Mot_n(a_1+a_2,b)}\wt(P)=\sum_{Q\in\Mot_n(\{a_1,a_2\},b)}\wt(Q)\\
=&\sum_{k=0}^{n}\binom{n}{n-k}a_2^{n-k}\sum_{R\in\Mot_k(a_1,b)}\wt(R)\\
=&\sum_{k=0}^{n}\binom{n}{n-k}a_2^{n-k}M_k^{(a_1,b)}.
\end{align*}
\end{proof}
\begin{remark}
There are some easy consequences of Lemma \ref{Mna12L}. Since $M_n^{(ax, bx^2)}=x^nM_n^{(a, b)}$ implies $M_n^{(a,b)}=(-1)^nM_n^{(-a,b)}$, by Lemma~\ref{Mna12L}  we have
\begin{equation}
M_n^{(a,b)}=M_n^{(2a-a, b)}=\sum_{k=0}^n\binom{n}{k}M_k^{(-a,b)}(2a)^{n-k}=(-1)^n\sum_{k=0}^n\binom{n}{k}M_k^{(a,b)}(-2a)^{n-k}.
\end{equation}
In particular, since $M_{n}=M_n^{(1,1)}$, $C_{n+1}=M_n^{(2,1)}$ and $S_{n+1}=2M_n^{(3,2)}$, by specializing $a$ and $b$ to be the corresponding values in the parenthesis we get
\begin{align}
&M_n=(-1)^n\sum_{k=0}^n(-2)^{n-k}\binom{n}{k}M_k,\\
&C_{n+1}=(-1)^n\sum_{k=0}^n(-4)^{n-k}\binom{n}{k}C_{k+1},\\
&S_{n+1}=(-1)^n\sum_{k=0}^n(-6)^{n-k}\binom{n}{k}S_{k+1}.\label{SSbinom}
\end{align}
\end{remark}
Now we are ready to prove Theorems \ref{CMCorT} and \ref{SMCorT}. From (\ref{Cnab}) and (\ref{CM}) we know that
\begin{align*}
\sum_{k=0}^{n-1}\frac{1}{n}{n \choose k}{n \choose k+1} x^{2k} b^{n-1-k}=
b^{-1}C_n^{(b,x^2)}=M_{n-1}^{(x^2+b,bx^2)}.
\end{align*}

Using Lemma \ref{Mna12L} and the identity $M_n^{(ax, bx^2)}=x^nM_n^{(a, b)}$ again, we get
\begin{align*}
M_{n-1}^{(x^2+b,bx^2)}=&\sum_{k=0}^{n-1}\binom{n-1}{k} M_k^{(ax,bx^2)}(x^2-ax+b)^{n-k}\\
=&
\sum_{k=0}^{n-1}\binom{n-1}{k} M_k^{(a,b)}x^k(x^2-ax+b)^{n-k}.
\end{align*}

Therefore Theorems \ref{CMCorT} is obtained.

Similarly, from (\ref{Snab}) and (\ref{SM}), we have
\begin{align*}
\sum_{k=0}^{n}\binom{n+k}{2k}C_k x^{2k}(b-x^2)^{n-k}=&S_n^{(b-x^2,x^2)}=bM_{n-1}^{(b+x^2,bx^2)}\\
=&b\sum_{k=0}^{n-1}\binom{n-1}{k}M_k^{(ax,bx^2)}(x^2-ax+b)^{n-1-k}\\
=&b\sum_{k=0}^{n-1}\binom{n-1}{k}M_k^{(a,b)}x^k(x^2-ax+b)^{n-1-k}.
\end{align*}
Therefore Theorem \ref{SMCorT} holds.
\qed

\section{More consequences of Theorems \ref{CMCorT}-\ref{thm:CSM}}
\setcounter{equation}{0}\setcounter{theorem}{0}

In this section, we shall list more identities involving $C_n$, $M_n$ and $S_n$, which follows from Theorems \ref{CMCorT}-\ref{thm:CSM}.

\begin{proposition} For $a,b\in\bC$,
\begin{align}
&\sum_{k=0}^{n-1}\binom{n-1}{k}C_{k+1}^{(a,b)}x^{k}(x-a)^{n-1-k}(x-b)^{n-1-k}\notag\\
=&\sum_{k=0}^{n-1}\frac{1}{n}\binom{n}{k}\binom{n}{k+1}x^{2k}a^{n-k}b^{n-1-k}\label{CCab}\\
=&\frac{1}{b}\sum_{k=0}^{n}\binom{n+k}{2k}C_k x^{2k}(ab-x^2)^{n-k}.\label{CCab2}
\end{align}
\end{proposition}
\begin{proof}
By (\ref{CM}) and (\ref{CMCor}), we have
\begin{align*}
&\sum_{k=0}^{n-1}\binom{n-1}{k}C_{k+1}^{(a,b)}x^{k}(x-a)^{n-1-k}(x-b)^{n-1-k}\\
 \xlongequal{(\ref{CM})}&a\sum_{k=0}^{n-1}\binom{n-1}{k}M_{k}^{(a+b,ab)}x^{k}\big(x^2-(a+b)x+ab\big)^{n-1-k}\\
\xlongequal{(\ref{CMCor})}&\sum_{k=0}^{n-1}\frac{1}{n}\binom{n}{k}\binom{n}{k+1}x^{2k}a^{n-k}b^{n-1-k}.
\end{align*}
Similarly, (\ref{CCab2}) can be obtained from  (\ref{CM}) and (\ref{SMCor}).
\end{proof}
In particular, substituting  $a=b=1$ in (\ref{CCab2}), we can get
\begin{equation}\label{CMCorCat}
\sum_{k=0}^{n}\binom{n+k}{2k}C_k x^{2k}(1-x^2)^{n-k}=\sum_{k=0}^{n-1}\binom{n-1}{k}C_{k+1}x^{k}(x-1)^{2n-2-2k}.
\end{equation}

By (\ref{SC}), which is $S_{n}^{(a,b)}=\frac{(a+b)}{b}C_{n}^{(b, a+b)}$, we may rewrite (\ref{CCab}) and (\ref{CCab2}) respectively as follows:
\begin{align}\label{SCab}
&\sum_{k=0}^{n-1}\binom{n-1}{k}S_{k+1}^{(a,b)}x^{k}(x-b)^{n-1-k}(x-a-b)^{n-1-k}\notag\\
=&\sum_{k=0}^{n-1}\frac{1}{n}\binom{n}{k}\binom{n}{k+1}x^{2k}(a+b)^{n-k}b^{n-1-k},
\end{align}
\begin{align}\label{SCab2}
&\sum_{k=0}^{n-1}\binom{n-1}{k}S_{k+1}^{(a,b)}x^{k}(x-b)^{n-1-k}(x-a-b)^{n-1-k}\notag\\
=&\frac1b\sum_{k=0}^{n}\binom{n+k}{2k}C_k x^{2k}(ab+b^2-x^2)^{n-k}.
\end{align}
Substituting $a=b=1$ into (\ref{SCab}) and (\ref{SCab2}), we have
\begin{align}
&\sum_{k=1}^{n-1}\binom{n-1}{k}S_{k+1}x^{k}(x-1)^{n-1-k}(x-2)^{n-1-k}\notag\\=&\sum_{k=0}^{n-1}\frac{1}{n}\binom{n}{k}\binom{n}{k+1}x^{2k} 2^{n-k}\label{CMCorSch}\\
=&\sum_{k=0}^{n}\binom{n+k}{2k}C_k x^{2k}(2-x^2)^{n-k}.
\end{align}
In particular, setting $x=\sqrt{2}$ in (\ref{CMCorSch}), we can get
\begin{equation}\label{SCsqrt2}
C_n=\frac{(-\sqrt{2})^{n-1}}{2^n}\sum_{k=1}^{n-1}(-1)^{k}\binom{n-1}{k}(\sqrt{2}-1)^{2(n-1-k)}S_{k+1}.
\end{equation}
Similarly, applying (\ref{CMCorCat}) with $x=\frac12\sqrt{2}$, we have
\begin{equation}\label{SCsqrt2b}
S_n=\sum_{k=0}^{n-1}2^{\frac k2+1}\binom{n-1}{k}(\sqrt{2}-1)^{2(n-1-k)}C_{k+1}.
\end{equation}
Of course, (\ref{SCsqrt2b}) also can be easily deduced from (\ref{SCsqrt2}) via a binomial transform.
\begin{proposition} For $\alpha,\beta\in\bC$,
\begin{align}
M_n^{(\alpha,\beta)}=&
\sum_{k=0}^{n+1}\bigg(\frac{\alpha+\sqrt{\alpha^2-4\beta}}2\bigg)^{n-2k}\cdot\frac{\beta^k}{n+1}\binom{n+1}{k}\binom{n+1}{k+1}\label{MCalphabeta}\\
=&\frac{(\alpha^2-4\beta)^{\frac{n+1}{2}}}{\beta}\sum_{k=0}^{n+1}\bigg(\frac{\alpha}{\sqrt{\alpha^2-4\beta}}-1\bigg)^{k+1}\cdot\frac{C_k}{2^{k+1}}\binom{n+k}{2k}.\label{MSalphabeta}
\end{align}
\end{proposition}
\begin{proof}
Let $a=\frac12(\alpha+\sqrt{\alpha^2-4\beta})$ and $b=\frac12(\alpha-\sqrt{\alpha^2-4\beta})$.
Evidently $a+b=\alpha$ and $ab=\beta$. So by (\ref{CM})
\begin{align*}
M_n^{(\alpha,\beta)}=&\frac{C_{n+1}^{(a,b)}}{a}=a^{n}\sum_{k=0}^{n+1}\frac1{n+1}\binom{n+1}{k}\binom{n+1}{k+1}\cdot\bigg(\frac ba\bigg)^k
\notag\\
=&\sum_{k=0}^{n+1}\bigg(\frac{\alpha+\sqrt{\alpha^2-4\beta}}2\bigg)^{n-2k}\cdot\frac{\beta^k}{n+1}\binom{n+1}{k}\binom{n+1}{k+1}.
\end{align*}
Similarly, letting  $a=\sqrt{\alpha^2-4\beta}$ and $b=\frac12(\alpha-\sqrt{\alpha^2-4\beta})$, we have
$a+2b=\alpha$ and $ab+b^2=\beta$. Then
\begin{align*}
M_n^{(\alpha,\beta)}=&\frac{S_{n+1}^{(a,b)}}{a+b}=\frac{1}{a+b}\sum_{k=0}^{n+1}\binom{n+k}{2k}C_k\cdot a^{n-k}b^k
\notag\\
=&\frac{(\alpha^2-4\beta)^{\frac{n+1}{2}}}{\beta}\sum_{k=0}^{n+1}\bigg(\frac{\alpha}{\sqrt{\alpha^2-4\beta}}-1\bigg)^{k+1}\cdot\frac{C_k}{2^{k+1}}\binom{n+k}{2k}.
\end{align*}
\end{proof}
In particular, substituting $\alpha=\beta=1$ in (\ref{MCalphabeta}) and (\ref{MSalphabeta}), we obtain
\begin{align}
M_n=&
\sum_{k=0}^{n+1}\bigg(\frac{1}{2}+\frac{\sqrt{3}}{2}\fii\bigg)^{n-2k}\cdot\frac{1}{n+1}\binom{n+1}{k}\binom{n+1}{k+1}\\
=&\bigg(\frac{1}{2}-\frac{\sqrt{3}}{2}\fii\bigg)^{n+2}
\sum_{k=0}^{n+1}\bigg(-\frac{3}{2}+\frac{\sqrt{3}}{2}\fii\bigg)^{n+1-k}\binom{n+1+k}{2k} C_{k},
\end{align}
where $\fii=\sqrt{-1}$.

Finally, let us see some applications of (\ref{Mna12}).
\begin{proposition}\label{MCSnk}
\begin{align}
M_{n}=&\sum_{k=0}^n(-1)^{n-k}\binom{n}{k}C_{k+1}\label{MCnk1}\\
=&\sum_{k=0}^n(-1)^k3^{n-k}\binom{n}{k}C_{k+1}\label{MCnk2}\\
=&\frac1{(\sqrt{2})^{n+2}}\sum_{k=0}^n(\sqrt{2}-3)^{n-k}\binom{n}{k}S_{k+1}\label{MSnk1}\\
=&\frac1{(\sqrt{2})^{n+2}}\sum_{k=0}^n(-1)^k(3+\sqrt{2})^{n-k}\binom{n}{k}S_{k+1}.\label{MSnk2}
\end{align}
\end{proposition}
\begin{proof}
By (\ref{Mna12}), we have
$$
M_n^{(1,1)}=\sum_{k=0}^n(-1)^{n-k}\binom{n}{k}M_k^{(2,1)}
=\sum_{k=0}^n(-1)^{n-k}\binom{n}{k}C_{k+1}.
$$
This concludes (\ref{MCnk1}). Similarly, (\ref{MCnk2}) follows from that
\begin{align*}
M_{n}&=(-1)^nM_n^{(-1,1)}=(-1)^nM_n^{(2-3,1)}\\
&=(-1)^n\sum_{k=0}^n\binom{n}{k}M_k^{(2,1)}(-3)^{n-k}
=\sum_{k=0}^n(-1)^k3^{n-k}\binom{n}{k}C_{k+1}.
\end{align*}
On the other hand, since $M_k^{(3,2)}=\frac12S_{k+1}$,
\begin{align*}
M_{n}&=\frac1{(\sqrt{2})^{n}}M_n^{(\sqrt{2},2)}=\frac1{(\sqrt{2})^{n}}M_n^{(3+\sqrt{2}-3,2)}\\
     &=\frac1{(\sqrt{2})^{n}}\sum_{k=0}^n\binom{n}{k}M_k^{(3,2)}(\sqrt{2}-3)^{n-k}\\
     &=\frac1{(\sqrt{2})^{n+2}}\sum_{k=0}^n(\sqrt{2}-3)^{n-k}\binom{n}{k}S_{k+1}.
\end{align*}
Also, we have
\begin{align*}
M_{n}&=\frac{M_n^{(-\sqrt{2},2)}}{(-\sqrt{2})^{n}}=\frac{M_n^{(3-\sqrt{2}-3,2)}}{(-\sqrt{2})^{n}}\\
     &=\frac1{(\sqrt{2})^{n+2}}\sum_{k=0}^n(-1)^k\binom{n}{k}(\sqrt{2}+3)^{n-k}S_{k+1}.
\end{align*}
Therefore~(\ref{MSnk1}) and~(\ref{MSnk2}) hold.
\end{proof}
Using the fact $S_n=2M_{n-1}^{(3,2)}$ again, together with the identities $M_n^{(\sqrt{2}, 2)}=2^{\frac n2}M_n^{(1,1)}$ and $M_n^{(-\sqrt{2}, 2)}=(-1)^{n}M_n^{(\sqrt{2}, 2)}$,  we can also get
\begin{align}
S_{n}&=2M_{n-1}^{(3,2)}=2M_{n-1}^{(3-\sqrt{2}+\sqrt{2},2)}\nonumber\\
     &=\sum_{k=0}^{n-1}2^{\frac k2+1}(3-\sqrt{2})^{n-1-k}\binom{n-1}{k}M_{k}\label{SMnk1}
\end{align}
\begin{align}
S_{n}&=2M_{n-1}^{(3,2)}=2M_{n-1}^{(3+\sqrt{2}-\sqrt{2},2)}\nonumber\\
     &=\sum_{k=0}^{n-1}(-1)^k2^{\frac k2+1}(3+\sqrt{2})^{n-1-k}\binom{n-1}{k}M_{k}.\label{SMnk2}
\end{align}
\begin{remark} There are two arithmetic consequences of Theorem \ref{MCSnk}. In view of (\ref{MCnk2}), we have
\begin{equation}
M_n\equiv(-1)^n\sum_{k=0}^{m-1}(-3)^k\binom{n}{k}C_{n-k+1}\pmod{3^m}
\end{equation}
for each $m\geq 1$. In particular,
\begin{equation}
M_n\equiv(-1)^nC_{n+1}\pmod{3}.
\end{equation}
Moreover, since $7=(3+\sqrt{2})(3-\sqrt{2})$, by (\ref{MSnk1}), we can get
\begin{equation}
M_n\equiv2^{n+2}S_{n+1}\pmod{7}.
\end{equation}
\end{remark}

\section{A bijection on the set of Catalan paths}
\setcounter{equation}{0}\setcounter{theorem}{0}

Assume that $P=p_1p_2\ldots p_{2n}$ is a Catalan path of order $n$ and $a,b\in\bC$.
For a Catalan path $P=p_1p_2\ldots p_{2n}$, a step $p_i$ is said to be in an even (resp. odd) position if $i$ is even (resp. odd).
We shall assign the steps of $P$ in even positions with some weights. For $1\leq i\leq n$, assign the step $p_{2i}$ with weight $a$ or $b$ according to whether $p_{2i}=\E$ or $\N$. The weighted Catalan path obtained in this way is called the even-north type $(a,b)$-Catalan path.
Let $\CAT_n(a,b)$ denote the set of all  even-north type $(a,b)$-Catalan paths of order $n$ and $\cC_n^{(a,b)}$ the weight of $\CAT_n(a,b)$. Let $\enor(P)$ denote the number of $\N$ steps in even positions on a Catalan path $P$. It's been proved that the statistic $\enor$ is distributed by the Narayana numbers~\cite{CYY08, DV84} and hence we have
\begin{align*}
\cC_n^{(a,b)}:=&\sum_{P\in \CAT_n(a,b)}\wt(P)\\
=&\sum_{P\in \CAT_n(a,b)}a^{n-\enor(P)}b^{\enor(P)}\\
=&\sum_{k=0}^{n-1}\frac1n\binom{n}{k}\binom{n}{k+1}a^{n-k}b^k.
\end{align*}
We can easily see from the algebraic expression that the following theorem holds.
\begin{theorem}\label{thmCC}
\begin{equation}\label{CabCab}
\cC_n^{(a,b)}=C_n^{(a,b)}.
\end{equation}
\end{theorem}
However, we are more interested in finding a purely combinatorial proof. That is, we  shall construct a bijection between two different types of $(a,b)$-Catalan paths since we believe it will be also meaningful on its own sense. To give a bijective proof of (\ref{CabCab}),  we first show the following lemma.
\begin{lemma}\label{Cat11}
For any $0\leq k\leq n-1$, there is a bijection between the set
$$
A_k:=\{P\in\Cat_n(a,b):\, \valley(P)=k\}
$$
and the set
$$
B_k:=\{P\in\CAT_n(a,b):\, \enor(P)=k\}.
$$
\end{lemma}
\begin{proof}

Define $\phi: A_k\rightarrow B_k$ as follows:

Step 1. Notice that each path $P\in\Cat(n)$ starting from $(0,0)$ and ending at $(n,n)$ always contains $n$ $\N$ and $\E$ steps respectively.
We can assume that $p_{s_1},p_{s_2},\ldots,p_{s_n}$ are the successive $\N$ steps of $P$ from bottom to top,
and $p_{t_1},p_{t_2},\ldots,p_{t_n}$ the successive $\E$ steps of $P$ successively from left to right.

For example, Figure~\ref{a Catalan path of order $6$ with $3$ valleys.} gives a Catalan path $P=\N\E\N\N\E\N\E\E\N\N\E\E$ of order $6$, whose $\N$ steps are $p_1,p_3,p_4,p_6,p_9,p_{10}$ and $\E$ steps are $p_2,p_5,p_7,p_8,p_{11},p_{12}$. So we have
\begin{align*}
P&=p_1p_2p_3p_4\ldots p_{11}p_{12}\\
 &=\N\E\N\N\E\N\E\E\N\N\E\E \\
 &=p_{s_1}p_{t_1}p_{s_2}p_{s_3}p_{t_2}p_{s_4}p_{t_3}p_{t_4}p_{s_5}p_{s_6}p_{t_5}p_{t_6}.
\end{align*}
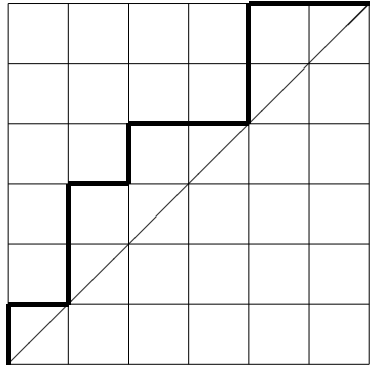
\begin{figure}[h]
\setlength{\unitlength}{1.0cm}
\begin{center}
\scalebox{0.8}{\begin{picture}(6,6)
\linethickness{0.075mm}
\multiput(0,0)(1,0){7}%
{\line(0,1){6}}
\multiput(0,0)(0,1){7}%
{\line(1,0){6}}
\put(0,0){\line(1,1){6}}
\linethickness{0.7mm}
\put(0,0){\line(0,1){1}}
\put(0,1){\line(1,0){1}}
\put(1,1){\line(0,1){2}}
\put(1,3){\line(1,0){1}}
\put(2,3){\line(0,1){1}}
\put(2,4){\line(1,0){2}}
\put(4,4){\line(0,1){2}}
\put(4,6){\line(1,0){2}}
\end{picture}}
\end{center}
\caption{a Catalan path of order $6$ with $3$ valleys.}
\label{a Catalan path of order $6$ with $3$ valleys.}
\end{figure}

Step 2. Clearly each valley in $P$ must be formed by some $p_{t_i}$ and $p_{s_j}$ with $s_j=t_i+1$.
Change $p_{t_i}$ by an $\N$ step and change $p_{s_j}$ by an $\E$ step, i.e.,
replace the valley $p_{t_i}p_{s_j}$ by a peak.
Do such a transformation for all valleys in $P$, and denote the resultant path by $\hat{P}=\hat{p}_1\hat{p}_2\cdots\hat{p}_{2n}$.

For example, for the Catalan path $P$ in Figure~\ref{a Catalan path of order $6$ with $3$ valleys.}, there are three valleys $p_2p_3$, $p_5p_6$, $p_8p_9$ in $P$.
So $$\hat{P}=\hat{p}_{1}\hat{p}_{2}\hat{p}_{3}\hat{p}_{4}\cdots
\hat{p}_{11}\hat{p}_{12}=\N\underline{\N\E}\N\underline{\N\E}\E\underline{\N\E}\N\E\E.$$

Step 3. Let $$\phi(P)=\hat{p}_{s_1}\hat{p}_{t_1}\hat{p}_{s_2}\hat{p}_{t_2}\cdots
\hat{p}_{s_n}\hat{p}_{t_n}.$$
That is, $\phi(P)$ is obtained by alternately placing $\hat{p}_{s_1},\ldots,\hat{p}_{s_n}$ and
$\hat{p}_{t_1},\ldots,\hat{p}_{t_n}$.

For example in Figure~\ref{a Catalan path of order $6$ with $3$ valleys.}, we have
$$
\phi(P)=\hat{p}_{s_1}\hat{p}_{t_1}\hat{p}_{s_2}\hat{p}_{t_2}\cdots
\hat{p}_{s_6}\hat{p}_{t_6}=\N \N\E\N\N\E\E\N\E\E\N\E.
$$
Clearly $\phi(P)$ has three $\N$ steps in even positions, which are the $2^{nd}$, $4^{th}$ and $8^{th}$ steps respectively as shown in Figure~\ref{a Catalan path of order $6$ with $3$ EN-steps.}.

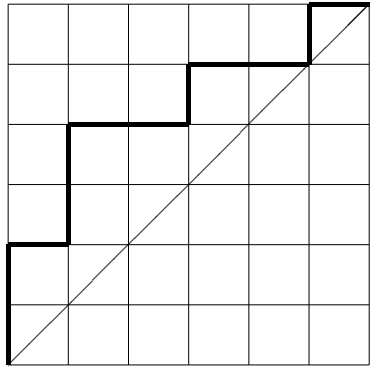
\begin{figure}[h]
\setlength{\unitlength}{1.0cm}
\begin{center}
\scalebox{0.8}{\begin{picture}(6,6)
\linethickness{0.075mm}
\multiput(0,0)(1,0){7}%
{\line(0,1){6}}
\multiput(0,0)(0,1){7}%
{\line(1,0){6}}
\put(0,0){\line(1,1){6}}
\linethickness{0.7mm}
\put(0,0){\line(0,1){2}}
\put(0,2){\line(1,0){1}}
\put(1,2){\line(0,1){2}}
\put(1,4){\line(1,0){2}}
\put(3,4){\line(0,1){1}}
\put(3,5){\line(1,0){2}}
\put(5,5){\line(0,1){1}}
\put(5,6){\line(1,0){1}}
\end{picture}}
\end{center}
\caption{a Catalan path of order $6$ with $3$ EN-steps.}
\label{a Catalan path of order $6$ with $3$ EN-steps.}
\end{figure}

We may write $\phi(P)=q_1q_2\cdots q_{2n}$.
It is not difficult to see that $\wt(\phi(P))=\wt(P)$ and
$\enor(\phi(P))=\valley(P)$. However,
we still need to show that $\phi(P)$ is a Catalan path. That is equivalent to prove that for any $1\leq h\leq 2n$, the following inequality holds:
\begin{equation}\label{PhiCat}
|\{1\leq i\leq h:\,q_i=\N\}|\geq
|\{1\leq j\leq h:\,q_j=\E\}|.
\end{equation}
Clearly we always have $q_{2i-1}=\N$ and $q_{2i}=\E$, unless $q_{2i-1}=\E$ or $q_{2i}=\N$.
Since $q_{2j-1}=\hat{p}_{s_j}$, if $q_{2j-1}=\E$, then $p_{s_{j}-1}p_{s_j}$ must be a valley in $P$, i.e., $t_i=s_j-1$ for some $1\leq i\leq n$. Since $P$ is a Catalan path, we must have $j\geq i+1$. Thus for any $1\leq j\leq n$ with $q_{2j-1}=\E$, there exists a unique $i$ such that $1\leq i\leq j-1$ and $q_{2i}=\hat{p}_{t_i}=\N$. Hence the number of $\N$ steps is always not less than the number of $\E$ steps in $\phi(P)$. Therefore (\ref{PhiCat}) is valid.

Conversely, in order to show that $\phi$ is a bijection, we need to prove that for each $Q\in B_k$, there exists a unique $P\in A_k$ such that
$Q=\phi(P)$. To do that, we orderly pair each evenly positioned $\N$ step and each oddly positioned $\E$ step of $Q$ together to get all valley points in $P$. A valley point is the vertex common to both steps of the valley. Explicitly, assume that $Q=q_1q_2\ldots q_{2n}\in B_k$ where $q_{2s_1-1},\ldots,q_{2s_k-1}=\E$ and $q_{2t_1},\ldots,q_{2t_k}=\N$. Then the valley points are formed by $(t_1, s_1-1), (t_2, s_2-1),\ldots, (t_k, s_k-1)$.
Since $Q$ is a Catalan path, we must have $t_i<s_i$ for each $1\leq i\leq k$.
We construct $P=p_1p_2\cdots p_{2n}$ as follows:

First, set $p_1=\cdots=p_{s_1-1}=\N$,
$p_{s_1}=\cdots=p_{s_1+t_1-1}=\E$ and $p_{s_1+t_1}=\N$. The first valley is formed by steps $p_{s_1+t_1-1}p_{s_1+t_1}$.
Next, set $p_{s_1+t_1+1},\ldots, p_{s_2+t_1-1}=\N$,
$p_{s_2+t_1},\ldots,p_{s_2+t_2-1}=\E$ and $p_{s_2+t_2}=\N$. The second valley is formed by steps $p_{s_2+t_2-1}p_{s_2+t_2}$.
Keep this process, until we get the last valley formed by $p_{t_k+s_k-1}p_{t_k+s_k}$.
Finally, set $p_{s_k+t_k+1},\ldots,p_{n+t_k}=\N$ and
$p_{n+t_k+1},\ldots,p_{2n}=\E$.

Note that for each $1\leq i\leq k$, we have
$$
|\{1\leq j<s_i+t_i:\,p_j=\N\}|-
|\{1\leq j<s_i+t_i:\,p_j=\E\}|=(s_i-1)-t_i\geq 0.
$$
So $P$ is also a Catalan path. It is not difficult to check that $\phi(P)=Q$. Hence $\phi$ is surely a bijection.

For example in Figure~\ref{a Catalan path of order $6$ with $3$ EN-steps.}, since $q_3=q_7=q_9=\E$ and $q_2=q_4=q_8=\N$, the valley points are $(1,1), (2,3), (4,4)$. Therefore, we have $p_1=\N$, $p_2=\E$, $p_3=p_4=\N$, $p_5=p_6=\N$, $p_7=p_8=\E$, $p_9=p_{10}=\N$, $p_{11}=p_{12}=\E$, which form exactly the path $P$ in Figure~\ref{a Catalan path of order $6$ with $3$ valleys.}.
\end{proof}
Notice that from the proof of Lemma~\ref{Cat11}, we have $\wt(\phi(P))=\wt(P)$ for each $P\in A_k$. So (\ref{CabCab}) immediately follows from Lemma \ref{Cat11}.

\section{Proofs of (\ref{CM}) and (\ref{SC})}
\setcounter{equation}{0}\setcounter{theorem}{0}

\begin{proof}[Proof of (\ref{CM})]

By Theorem~\ref{thmCC}, to show (\ref{CM}) is equivalent to prove the following identity
\begin{equation}\label{CM2}
\cC_n^{(a,b)}=aM_{n-1}^{(a+b,ab)}.
\end{equation}

Let $\Mot_n(\{a,b\},a,b)$ denote the set of all weighted $n$-ordered Motzkin paths, with each $\E_2$ step  assigned with weight $a$, each $\N_2$ step assigned with weight $b$, and each $\D$ step assigned with weight either $a$ or $b$.
For a path $P\in \Mot_{n-1}(\{a,b\},a,b)$, we can remove the weights of $\N_2$ step and reassign each
$\E_2$ step with weight $ab$. Then we get an $(\{a,b\},ab)$-Motzkin path of order $n-1$.
Evidently the two paths have the same weight, since the number of $\N_2$ steps and $\E_2$ steps in a Motzkin path are always equal.
So we have
$$
\sum_{P\in\Mot_{n-1}(\{a,b\},a,b)}\wt(P)=\sum_{P\in\Mot_{n-1}(\{a,b\},ab)}\wt(P).
$$
On the other hand, from the discussion in the proof of Lemma \ref{Mna12L}, we have
$$
M_{n-1}^{(a+b,ab)}=\sum_{P\in\Mot_{n-1}(a+b,ab)}\wt(P)
=\sum_{P\in\Mot_{n-1}(\{a,b\},ab)}\wt(P).
$$

It suffices to give a bijection $\psi:\,\CAT_n(a,b)\to\Mot_{n-1}(\{a,b\},a,b)$.
For a path $P=p_1p_2\ldots p_{2n}\in \CAT_n(a,b)$, we know that $p_1=\N$ and $p_{2n}=\E$ with weight $a$.
Let $\psi(P)=q_1q_2\ldots q_{n-1}\in \Mot_n(\{a,b\},a,b)$ be  given as follows: for $1\leq i\leq n-1$,  \medskip

$p_{2i}p_{2i+1}=\N\N$ with weight $b$ if and only if  $q_{i}=\N_2$  with weight $b$;
\medskip

$p_{2i}p_{2i+1}=\E\E$ with weight $a$ if and only if  $q_{i}=\E_2$  with weight $a$;
\medskip

$p_{2i}p_{2i+1}=\N\E$ with weight $b$ if and only if  $q_{i}=\D$  with weight $b$;
\medskip

$p_{2i}p_{2i+1}=\E\N$ with weight $a$ if and only if  $q_{i}=\D$  with weight $a$.
\medskip

Clearly $\psi$ is a bijection. Since the last step $p_{2n}$ is always $\E$ with weight $a$, we have $\wt(P)=a\cdot \wt(\psi(P))$.
Therefore, we have
$$
\cC_n^{(a,b)}=\sum_{P\in\CAT_n(a,b)}\wt(P)=a\bigg(\sum_{\psi(P)\in\Mot_{n-1}(\{a,b\},a,b)}\wt(\psi(P))\bigg)
=aM_{n-1}^{(a+b,ab)}.
$$
\end{proof}

\begin{proof}[Proof of (\ref{SC})]
A peak in a Catalan path is defined to be an $\N$ step followed immediately by an $\E$ step.
Assume that $P=p_1p_2\cdots p_{2n}$ is a Catalan path of order $n$.
For $1\leq i\leq n-1$, assign each $\E$ step $p_i$ with weight $b$
provided that $p_{i}p_{i+1}=\N\E$ which forms a peak in $P$.
Also, assign each of the rest $\E$ step of $P$ with weight $a$.
Thus we get a {\it peak type $(a,b)$-Catalan path}.
Let $\overline{\Cat}_n(a,b)$ be the set of all peak type $(a,b)$-Catalan paths of order $n$,
and $\peak(P)$ denote the number of all peaks in $P$. Clearly $\peak(P)=\valley(P)+1$.
Hence
\begin{align}\label{pvcatalan}
\sum_{P\in\overline{\Cat_n}(a,b)}\wt(P)=&
\sum_{P\in\overline{\Cat_n}(a,b)}a^{n-\peak(P)}b^{\peak(P)}\notag\\
=&
\sum_{P\in\Cat_n(a,b)}a^{n-1-\valley(P)}b^{\valley(P)+1}\notag\\
=&\frac{b}{a}\sum_{P\in\Cat_n(a,b)}a^{n-\valley(P)}b^{\valley(P)}=
\frac{b}{a}\cdot C_n^{(a,b)}.
\end{align}

For $P$ an $n$-ordered Catalan path, if we assign each $\E$ step of a peak with weight either $b_1$ or $b_2$, and each other $\E$ step of $P$ with weight $a$, then the resultant weighted path is called  a peak type $(a,\{b_1,b_2\})$-Catalan paths of order $n$.
Let $\overline{\Cat}_n(b,\{a,b\})$ denote the set of all peak type $(b,\{a,b\})$-Catalan paths of order $n$.
For a peak type $(b, a+b)$ Catalan path $P$,  by splitting the weight $a+b$ into the summation of $a$ and $b$,
we can get $2^{\peak(P)}$ peak type $(b,\{a,b\})$-Catalan paths.
So
\begin{align}\label{pabcatalan}
\sum_{P\in\overline{\Cat}_n(b,a+b)}\wt(P)=\sum_{P\in\overline{\Cat}_n(b,\{a,b\})}\wt(P).
\end{align}

In view of (\ref{pvcatalan}) and (\ref{pabcatalan}), to prove (\ref{SC}), we only need to find a bijection from $\Sch_n(a,b)$
to $\overline{\Cat}_n(b,\{a,b\})$. Such a bijection can be constructed in a natural way.
For an $(a,b)$-Schr\"{o}der path $P\in \Sch_n(a,b)$, since each $\D$ step in $P$ is assigned with weight $a$, we can then
replace each $\D$ step by an adjacent $\E \N$ pair and assign the $\E$ step of this peak with weight $a$. Doing this replacement for all $\D$ steps in the $(a,b)$-Schr\"{o}der path $P$. We get a peak type $(b,\{a,b\})$-Catalan path.
Conversely, for a peak type $(b,\{a,b\})$-Catalan path, we can replace each adjacent $\E \N$ pair by a $\D$ step with weight $a$, provided that the $\E$ step is assigned with weight $a$. We also can get an $(a,b)$-Schr\"oder path. Thus we surely obtain a bijection from $\Sch_n(a,b)$ to $\overline{\Cat}_n(b,\{a,b\})$. It follows that
$$
\sum_{P\in \Sch_n(a,b)}\wt(P)=\sum_{P\in\overline{\Cat}_n(b,\{a,b\})}\wt(P)=
\sum_{P\in\overline{\Cat}_n(b,a+b)}\wt(P)=
\frac{a+b}{b}\cdot C_n^{(b,a+b)},
$$
and therefore~(\ref{SC}) holds.
\end{proof}
\begin{remark}
It can be seen that (\ref{SM}) easily follows from (\ref{CM}) and~(\ref{SC}).
In fact, using the same way in~\cite{Y07}, we can also give a direct bijection proof of~(\ref{SM}).
However, since the proof is a little bit tedious and contains almost no new ideas, we omit it in this paper.
\end{remark}

\end{document}